\documentclass[11pt]{article}
\usepackage[a4paper,margin=29mm]{geometry}
\usepackage{amsmath,amssymb,amsthm,mathtools}
\usepackage{array,booktabs,microtype}
\usepackage[hidelinks]{hyperref}
\hypersetup{
 pdftitle={Maximum Energy of Integral Circulant Graphs with Same-Parity Exponents},
 pdfauthor={Jianwei Jiang and Chunhua Yang},
 pdfkeywords={integral circulant graph, graph energy, Ramanujan sum, weighted cut}
}
\allowdisplaybreaks[2]
\numberwithin{equation}{section}
\newtheorem{theorem}{Theorem}[section]
\newtheorem{lemma}[theorem]{Lemma}
\newtheorem{proposition}[theorem]{Proposition}

\theoremstyle{remark}

\newtheorem{theoremA}{Theorem}

\newcommand{\ICG}{\operatorname{ICG}}
\newcommand{\Energy}{\mathcal E}
\newcommand{\one}{\mathbf1}

\newcommand{\tr}{\operatorname{tr}}
\newcommand{\diag}{\operatorname{diag}}
\newcommand{\cut}{\operatorname{cut}}

\newcommand{\BB}{\mathsf B}
\newcommand{\PP}{\mathsf P}
\newcommand{\HH}{\mathsf H}
\newcommand{\EE}{\mathsf E}
\newcommand{\emb}{\mathcal J}
\newcommand{\Q}{\mathcal Q}
\newcommand{\Dodd}{\mathcal D_{\mathrm{odd}}}
\newcommand{\Deven}{\mathcal D_{\mathrm{even}}^{-}}
\newcommand{\ind}{\mathbf1}
\newcommand{\trans}{\mathsf T}
\title{Maximum Energy of Integral Circulant Graphs with Same-Parity Exponents}
\author{Jianwei Jiang\thanks{Email: \texttt{jianweijiangwf@sina.com}}\quad
Chunhua Yang\thanks{Corresponding author. Email: \texttt{chunhuayangwf@sina.com}}\\
\small School of Mathematics and Statistics, Weifang University,\\
\small Weifang, Shandong, China}
\date{}
\begin{document}
\maketitle

\begin{abstract}
The energy of a finite graph is the sum of the absolute values of the
eigenvalues of its adjacency matrix. Let $n=p^aq^b$, where $p,q\ge5$ are distinct primes
and $a,b\ge1$ have the same parity.
Among all nonempty sets $\mathcal D$ of proper divisors of $n$, the energy of the
corresponding integral circulant graph is maximised by exactly two divisor
sets: the divisors $p^iq^j$ with $i+j$ odd, and the divisors with $i+j$ even
after removing $p^aq^b$. The maximum energy is
\[
 \Energy_{\max}(n)
 =\frac12\left[
 n+d_a(p)d_b(q)-2g_a(p)g_b(q)
 \right],
\]
where
\[
 d_k(x)
 =(2k+1)x^k
 +4\sum_{j=0}^{k-1}(-1)^{k-j}(j+1)x^j,
 \qquad
 g_k(x)
 =\frac{x^k(x-1)+2(-1)^k}{x+1}.
\]

For the proof, weighted prime-power Ramanujan transforms yield an exact
tensor-product representation of the graph energy, reducing the extremal
problem to a weighted cut separating two prescribed vertices. After a sign
conjugation, the relevant transforms decompose into embedded path blocks; the
nontrivial blocks correspond to a weighted path together with $3$-chords of
weight $-1$. We prove a sharp lower bound for cuts separating the two
prescribed vertices in Kronecker products of these path blocks. For products
of two such nontrivial blocks, equality holds only for cuts that isolate one of
the two prescribed vertices. Combined with the energy-to-cut correspondence, these equality
cases yield exactly the two maximising divisor sets.
\end{abstract}



\section{Introduction}

Let $G$ be a finite graph with adjacency eigenvalues
$\lambda_1,\ldots,\lambda_n$. Its energy is
$\Energy(G)=\sum_{j=1}^{n}|\lambda_j|$; see Li, Shi and
Gutman~\cite{li2012energy}. For an integer $n>1$ and a nonempty set
$\mathcal D$ of proper divisors of $n$, let $\ICG_n(\mathcal D)$ be the
graph on $\mathbb Z/n\mathbb Z$ in which distinct vertices $u,v$ are
adjacent exactly when $\gcd(u-v,n)\in\mathcal D$. By
So~\cite{so2006integral}, these are precisely the integral circulant
graphs with at least one edge.

The extremal-energy problem for integral circulant graphs is best
understood in the prime-power setting. Sander and Sander obtained an
explicit formula for the energy and later gave a complete description of
the maximum energy together with all divisor sets attaining it
\cite{sander2011energy,sander2013exact}. For orders involving more than
one prime factor, the picture is less complete: Le and Sander
\cite{le2012extremal} used multiplicative divisor sets to obtain bounds
on the minimum and maximum energies. Earlier work of Ili\'c and
Ba\v{s}i\'c~\cite{ilicbasic2011energy} provided explicit formulas for the
energy of several families of integral circulant graphs and also addressed
the corresponding minimum-energy problem.

In our earlier work~\cite{jiangyang2026opposite}, we determined the maximum
energy of integral circulant graphs of order $n=p^{2r}q^{2s+1}$,
where $p,q$ are distinct odd primes, $q\ge5$, $r\ge1$, and $s\ge0$, and
proved that the even checkerboard divisor set is the unique maximiser.
In the notation of Section~\ref{sec:reduction}, with $a=2r$ and
$b=2s+1$, the proof is based on
\[
 2\Energy(\ICG_n(\mathcal D))
 =\|ne_t+Ky\|_1
 \le n+\|Ky\|_1
 \le n+d_a(p)d_b(q),
\]
where $d_k(x)$ is defined below. For opposite-parity exponents, both
inequalities can hold with equality. When $a$ and $b$ have the same parity,
equality in the second inequality is still attained by the corresponding
checkerboard vector, but for this vector $(Ky)_t<0$. Since $ne_t$ is positive only at
$t$, cancellation occurs in the last coordinate and the first inequality
is strict. Accordingly, in this paper we directly maximise
$\frac12\|ne_t+Ky\|_1$.

For $k\ge1$ and $x>1$, put
\begin{equation*}
 d_k(x)=(2k+1)x^k+4\sum_{j=0}^{k-1}(-1)^{k-j}(j+1)x^j,
 \qquad
 g_k(x)=\frac{x^k(x-1)+2(-1)^k}{x+1}.
\end{equation*}
The function $g_k$ is also a polynomial:
$g_k(x)=x^k+2\sum_{j=0}^{k-1}(-1)^{k-j}x^j$.
The quantities $d_k(x)$ and $g_k(x)$ are, respectively, the trace and
the last diagonal entry of the diagonal weight matrix used
in~\cite[Section~3.1]{jiangyang2026opposite}.

\begin{theoremA}\label{thm:main}
Let $p,q\ge5$ be distinct primes, and let $a,b\ge1$ be integers with
$a+b$ even. Put $n=p^aq^b$. Among all nonempty sets $\mathcal D$ of proper divisors of $n$, the
maximum energy of $\ICG_n(\mathcal D)$ is
\begin{equation}\label{eq:main-energy}
 \Energy_{\max}(n)
 =\frac12\left[
 n+d_a(p)d_b(q)-2g_a(p)g_b(q)
 \right].
\end{equation}
Exactly two divisor sets attain this value:
\begin{equation}\label{eq:maximising-sets}
\begin{aligned}
 \Dodd
 &=\{p^iq^j:0\le i\le a,\ 0\le j\le b,\ i+j\text{ odd}\},\\
 \Deven
 &=\{p^iq^j:0\le i\le a,\ 0\le j\le b,\ i+j\text{ even}\}
   \setminus\{p^aq^b\}.
\end{aligned}
\end{equation}
\end{theoremA}

The two sets in~\eqref{eq:maximising-sets} partition the proper divisors
of $n$, so the corresponding graphs are complementary. Together with
\cite{jiangyang2026opposite}, Theorem A determines the maximum energy and all
maximising divisor sets for every order $p^aq^b$, where $p,q\ge5$ are
distinct primes and $a,b\ge1$.

To analyse the cancellation retained in the exact objective,
we depart from the semidefinite argument
of~\cite{jiangyang2026opposite} while retaining its weighted
Ramanujan transform framework. After a sign conjugation,
we convert the problem into a weighted cut problem in which
two prescribed vertices lie on opposite sides of the cut. The
sign-conjugated prime-power transforms are decomposed into embedded
path blocks, together with an additional last-coordinate term in the
even-exponent case. Each nontrivial block is supported on a weighted path with
$3$-chords of weight $-1$. The contribution of each negative $3$-chord is
transferred to the three edges of its associated path, leaving
nonnegative correction terms and reducing the cut estimate to a residual
weighted product graph. We then establish a sharp lower bound and
characterise its equality cases for Kronecker products of two such blocks
(Theorem~\ref{thm:path-blocks}).
These blockwise estimates are lifted to the full transforms; in the
even-exponent case, the strict contribution of the outermost nontrivial block
controls the additional mixed terms. Combined with the energy-to-cut
correspondence, this yields both the maximum in~\eqref{eq:main-energy} and
the two maximising divisor sets in~\eqref{eq:maximising-sets}.

\section{From graph energy to a cut sum}\label{sec:reduction}

\subsection{The transforms and the divisor coordinates}

For $k\ge1$ and a prime $p$, as in \cite[Eq.~(2.2)]{jiangyang2026opposite}, define the
weighted prime-power Ramanujan transform $T_k(p)$ by
\begin{equation*}
 (T_k(p))_{ui}=\varphi(p^{k-u})c_{p^{k-i}}(p^u)
 \qquad(0\le u,i\le k),
\end{equation*}
where $\varphi$ is Euler's totient function and $c_h$ is the Ramanujan
sum. By the standard prime-power formula for Ramanujan sums
\cite[Section~8.3]{apostol1976}, the entries of $T_k(p)$ admit the following
explicit polynomial expressions. We use the same formulas to define
$T_k(x)$ for every real parameter $x>1$; see
\cite[Proposition~2.1]{jiangyang2026opposite}:
\begin{equation}\label{eq:T-entries}
 (T_k(x))_{ij}=
 \begin{cases}
  0,&i,j<k,\ i+j\le k-2,\\
  -x^{k-1}(x-1),&i,j<k,\ i+j=k-1,\\
  x^{2k-i-j-2}(x-1)^2,&i,j<k,\ i+j\ge k,
 \end{cases}
\end{equation}
with boundary entries
\begin{equation}\label{eq:T-boundary}
 (T_k(x))_{ik}=(T_k(x))_{ki}=x^{k-i-1}(x-1)\quad(i<k),
 \qquad (T_k(x))_{kk}=1.
\end{equation}
In particular, $T_k(x)$ is symmetric and
\begin{equation}\label{eq:factor-row-sum}
 T_k(x)\one=x^ke_k.
\end{equation}
Here and below, $e_k$ is the last standard basis vector and $\one$ is
an all-one vector of the required size.

Let $s_k=(1,-1,\ldots,(-1)^k)^{\trans}$ and
$S_k=\diag(s_k)$. We also use the diagonal matrix $\Delta_k(x)$
introduced in~\cite[Section~3.1]{jiangyang2026opposite}. The following
properties of $\Delta_k(x)$ will be used below; see
\cite[Lemma~3.1 and equations~(3.4), (3.6)--(3.7)]{jiangyang2026opposite}:
\begin{equation}\label{eq:factor-parity}
 T_k(x)s_k=(-1)^k\Delta_k(x)s_k,\qquad
 \tr\Delta_k(x)=d_k(x),\qquad
 (\Delta_k(x))_{kk}=g_k(x).
\end{equation}
For $x\ge5$, $\Delta_k(x)$ is a positive diagonal matrix; see
\cite[Section~3.1]{jiangyang2026opposite}.

Fix distinct primes $p,q\ge5$ and positive integers $a,b$, and set
\begin{equation*}
n=p^aq^b,\qquad
 \mathcal I=\{0,\ldots,a\}\times\{0,\ldots,b\},\qquad
 t=(a,b),\qquad N=(a+1)(b+1).
\end{equation*}
Vectors indexed by $\mathcal I$ are ordered by
$\iota(i,j)=i+(a+1)j$, so the first coordinate varies fastest.
A divisor set is encoded by a binary matrix $X=(x_{ij})$, where
$x_{ij}=1$ precisely when $p^iq^j\in\mathcal D$. Write
$x=\operatorname{vec}(X)$ in this order and put
\begin{equation}\label{eq:centred-vector}
 y=2x-\one,\qquad x=\frac12(\one+y).
\end{equation}
The proper-divisor condition is $x_t=0$, or equivalently $y_t=-1$.
We allow the empty divisor set in intermediate upper bounds; the
corresponding graph has energy zero.

Define
\begin{equation*}
\begin{gathered}
 K=T_b(q)\otimes T_a(p),\qquad
 \Delta=\Delta_b(q)\otimes\Delta_a(p),\\
 s=s_b\otimes s_a,\qquad S=S_b\otimes S_a=\diag(s),\qquad
 \tau=d_a(p)d_b(q),\qquad \gamma=g_a(p)g_b(q).
\end{gathered}
\end{equation*}
Thus $s_{i,j}=(-1)^{i+j}$ and $\Delta_{tt}=\gamma$.
With the above ordering, the standard vectorisation identity
$\operatorname{vec}(AXB^{\trans})=(B\otimes A)\operatorname{vec}(X)$ gives
$Kx=\operatorname{vec}\!\bigl(T_a(p)XT_b(q)^{\trans}\bigr)$. Hence, by the
energy formula in \cite[Proposition~2.2]{jiangyang2026opposite},
$\Energy(\ICG_n(\mathcal D))=\|Kx\|_1$.
Also, by~\eqref{eq:factor-row-sum},
\begin{equation}\label{eq:mean-term}
 K\one
 =(T_b(q)\one)\otimes(T_a(p)\one)
 =q^be_b\otimes p^ae_a=ne_t.
\end{equation}
Combining~\eqref{eq:centred-vector} and~\eqref{eq:mean-term} gives
\begin{equation}\label{eq:F}
 2\Energy(\ICG_n(\mathcal D))
 =\|K(\one+y)\|_1=\|ne_t+Ky\|_1=:F(y).
\end{equation}

\subsection{The prescribed terminals}

For $k\ge1$ and $x>1$ put
\begin{equation}\label{eq:B-definition}
 \BB_k(x)=(-1)^kS_kT_k(x)S_k.
\end{equation}
Since $S_k\one=s_k$, the first identity in~\eqref{eq:factor-parity} gives
\begin{equation}\label{eq:B-row-sum}
 \BB_k(x)\one
 =(-1)^kS_kT_k(x)s_k
 =S_k\Delta_k(x)s_k
 =\Delta_k(x)S_ks_k
 =\Delta_k(x)\one.
\end{equation}
Here $S_k$ and $\Delta_k(x)$ commute because both are diagonal, and
$S_ks_k=\one$. In particular, the last row sum of $\BB_k(x)$ is
$g_k(x)$.
We now use the assumption that $a+b$ is even. Put $C=SKS$. Then
\begin{equation}\label{eq:C}
 C=\BB_b(q)\otimes\BB_a(p),\qquad
 C\one=\Delta\one,\qquad
 \sum_{u,v\in\mathcal I}C_{uv}=\tau.
\end{equation}
The first equality follows from the mixed-product property and the scalar identity
$(\alpha A)\otimes(\beta B)=\alpha\beta(A\otimes B)$, together with the definition of
$\BB_k(x)$ and $(-1)^{a+b}=1$. The second
follows by right-multiplying the first by
$\one=\one_{b+1}\otimes\one_{a+1}$, using the mixed-product property and
\eqref{eq:B-row-sum}. Left-multiplying the second by $\one^{\trans}$, and
using that $\Delta$ is diagonal together with
$\tr(A\otimes B)=\tr(A)\tr(B)$, gives the third. Since $t=(a,b)$ is the last
index in the above ordering, taking the $t$-th coordinate of
$C\one=\Delta\one$ gives $\sum_{v\in\mathcal I}C_{tv}=\Delta_{tt}=\gamma$.
Thus the last row sum of $C$ is $\gamma$. Since $C$ is symmetric, its last
column sum is also $\gamma$.

Regard $C_{uv}$ as the weight of an edge from row vertex $u$ to column
vertex $v$. Row and column vertices are separate copies of $\mathcal I$;
even $C_{tt}$ joins two different vertices. For sign vectors
$\zeta,\xi\in\{\pm1\}^N$, define
\begin{equation}\label{eq:cut-definition}
 \cut_C(\zeta,\xi)
 =\sum_{u,v\in\mathcal I}C_{uv}\ind_{\{\zeta_u\ne\xi_v\}}.
\end{equation}
Some weights may be negative. The quantity in~\eqref{eq:cut-definition}
is the sum of the weights crossing the sign partition.
We consider cuts satisfying
\begin{equation}\label{eq:pins}
 \zeta_t=1,\qquad \xi_t=-1.
\end{equation}
The two vertices indexed by $t$ are called the terminals. The
\emph{terminal cuts} are the cuts that isolate one of them:
\begin{equation}\label{eq:terminal-cuts}
 (\zeta,\xi)=(-\one+2e_t,-\one)
 \quad\text{or}\quad
 (\zeta,\xi)=(\one,\one-2e_t).
\end{equation}
Since the entries in the last row and last column of $C$ each sum to $\gamma$,
both terminal cuts have weight $\gamma$.

\begin{lemma}\label{lem:energy-cut}
For a fixed $y\in\{\pm1\}^N$ with $y_t=-1$, put $\xi=Sy$. Then
\begin{equation}\label{eq:energy-cut}
 F(y)=n+\tau-2\min_{\substack{\zeta\in\{\pm1\}^N\\\zeta_t=1}}
                       \cut_C(\zeta,\xi).
\end{equation}
Consequently, if every cut satisfying~\eqref{eq:pins} has weight at
least $\gamma$, and equality occurs only in~\eqref{eq:terminal-cuts},
then
\begin{equation*}
 \max_{y_t=-1}F(y)=n+\tau-2\gamma,
\end{equation*}
and the only maximising vectors are $-s$ and $s-2e_t$.
\end{lemma}
\begin{proof}
By the entry formula for the Kronecker product and~\eqref{eq:T-boundary},
write $w_{i,j}:=K_{t,(i,j)}=(T_b(q))_{b,j}(T_a(p))_{a,i}
=\varphi(q^{b-j})\varphi(p^{a-i})>0$ for $0\le i\le a$ and
$0\le j\le b$. Taking the last coordinate of~\eqref{eq:mean-term} gives
$\sum_{(i,j)\in\mathcal I}w_{i,j}=n$. Hence, since
$y_{i,j}\in\{\pm1\}$ for $0\le i\le a$ and $0\le j\le b$,
\begin{equation}\label{eq:last-coordinate-positive}
 n+(Ky)_t=\sum_{(i,j)\in\mathcal I}w_{i,j}(1+y_{i,j})\ge0.
\end{equation}
In the identity $\|v\|_1=\max_{z\in\{\pm1\}^N} z^{\trans}v$, the maximum is attained by choosing $z_i$ to have the same sign as $v_i$
whenever $v_i\ne0$, with either choice allowed when $v_i=0$.
By~\eqref{eq:last-coordinate-positive}, the last coordinate of
$v=ne_t+Ky$ is nonnegative, so a maximising sign vector can always be
chosen with $z_t=1$. Hence
\begin{equation}\label{eq:sign-dual}
 F(y)=n+\max_{\substack{z\in\{\pm1\}^N\\z_t=1}}z^{\trans}Ky.
\end{equation}
Since $s_t=1$ and $S^2=I$, the substitutions $\zeta=Sz$ and $\xi=Sy$
give $\zeta_t=s_tz_t=1$ and $\xi_t=s_ty_t=-1$, so the constraints
become~\eqref{eq:pins}. Moreover, $z=S\zeta$ and $y=S\xi$, and hence,
by $C=SKS$, $z^{\trans}Ky=\zeta^{\trans}C\xi$. For
$\zeta_u,\xi_v\in\{\pm1\}$,
$\zeta_u\xi_v=1-2\ind_{\{\zeta_u\ne\xi_v\}}$.
Therefore, by~\eqref{eq:C} and~\eqref{eq:cut-definition},
\begin{equation}\label{eq:bilinear-cut}
\begin{aligned}
 \zeta^{\trans}C\xi
 &=\sum_{u,v}C_{uv}
       \bigl(1-2\ind_{\{\zeta_u\ne\xi_v\}}\bigr)\\
 &=\tau-2\cut_C(\zeta,\xi).
\end{aligned}
\end{equation}
Substituting~\eqref{eq:bilinear-cut} into~\eqref{eq:sign-dual} gives
\eqref{eq:energy-cut}.

Under the stated cut bound,~\eqref{eq:energy-cut} gives
$F(y)\le n+\tau-2\gamma$. Both terminal cuts attain equality. By~\eqref{eq:terminal-cuts}, their
column sign vectors are $-\one$ and $\one-2e_t$. Since $\xi=Sy$ and
$S^2=I$, these correspond respectively to $y=S(-\one)=-s$ and
$y=S(\one-2e_t)=s-2e_t$.

Conversely, if $y$ is a maximiser, then the minimum
in~\eqref{eq:energy-cut} equals $\gamma$. Under the above assumption, any
cut attaining this minimum must be one of the two terminal cuts, so
$\xi=-\one$ or $\xi=\one-2e_t$, equivalently, $y=-s$ or
$y=s-2e_t$.
Therefore these are the only maximisers.
\end{proof}

By Lemma~\ref{lem:energy-cut}, it remains to prove that every cut of $C$
satisfying~\eqref{eq:pins} has weight at least $\gamma$, with equality only
for the two terminal cuts in~\eqref{eq:terminal-cuts}. Since~\eqref{eq:C}
gives $C=\BB_b(q)\otimes\BB_a(p)$, we first decompose the factors
$\BB_k(x)$. The resulting nontrivial matrices will be represented by weighted
paths in Section~\ref{sec:path-cuts}, where the required cut bounds for their
Kronecker products are established.

\section{Decomposing a transform into path matrices}\label{sec:decomposition}

For a $(j+1)\times(j+1)$ matrix $A$ and $k\ge j$, let
$\emb_{k,j}(A)$ be the $(k+1)\times(k+1)$ matrix with $A$ in its
lower-right block on coordinates $k-j,\ldots,k$ and zero elsewhere.
In particular, the last row and last column of $A$ are embedded in the last row and last column of the larger matrix. Put $\EE_k=e_ke_k^{\trans}$ and define
\begin{equation}\label{eq:path-matrix}
 (\PP_k)_{ij}=\ind_{\{i+j=k-1\text{ or }i+j=k\}}
 \quad(0\le i,j\le k),\qquad \PP_0=[1].
\end{equation}
For $j\ge1$ and a real number $c$, set
\begin{equation}\label{eq:H-definition}
 \HH_j(c)=c\PP_j-\emb_{j,j-1}(\PP_{j-1}).
\end{equation}
For a square matrix $A$ with a designated last coordinate, write $h(A)$
for its last row sum. The definitions give
\begin{equation}\label{eq:path-row-sums}
 h(\PP_j)=1,\qquad h(\HH_j(c))=c-1,\qquad h(\EE_k)=1.
\end{equation}
The same sums hold for the last columns, because these matrices are
symmetric.

\begin{lemma}
Set $\BB_0(x)=[1]$. For $k\ge1$ and $x>1$,
\begin{equation}\label{eq:B-recurrence}
 \BB_k(x)=x^{k-1}(x-1)\PP_k
               -\emb_{k,k-1}(\BB_{k-1}(x)).
\end{equation}
Consequently,
\begin{equation}\label{eq:B-decomposition}
\begin{split}
 \BB_k(x)={}&
 \sum_{\substack{j=k,k-2,\ldots\\j\ge2}}
       x^{j-2}(x-1)\emb_{k,j}(\HH_j(x))\\
 &+\begin{cases}
    \emb_{k,1}(\HH_1(x-1)),&k\text{ odd},\\
    \EE_k,&k\text{ even}.
   \end{cases}
\end{split}
\end{equation}
\end{lemma}
\begin{proof}
By~\eqref{eq:B-definition}, each entry of $\BB_k$ is
$(-1)^{k+i+j}(T_k)_{ij}$. For $1\le i,j\le k$, the $(i,j)$-entry of
$\emb_{k,k-1}(\BB_{k-1})$ is $(\BB_{k-1})_{i-1,j-1}$, while its
entries in row $0$ and column $0$ are zero. We verify that
$\BB_k+\emb_{k,k-1}(\BB_{k-1})$ has the two constant antidiagonals
on the right of~\eqref{eq:B-recurrence}.

If $i+j\le k-2$, both entries are zero. If $i+j=k-1$, the embedded
entry is zero and~\eqref{eq:T-entries} gives
$(\BB_k)_{ij}=x^{k-1}(x-1)$.
If $i+j=k$ and $i,j<k$, the two entries add to
$x^{k-2}(x-1)^2+x^{k-2}(x-1)=x^{k-1}(x-1)$.
For the two positions $(0,k)$ and $(k,0)$, the same value follows
from~\eqref{eq:T-boundary}, and the embedded entry vanishes.

If $i+j\ge k+1$ and $i,j<k$, then by~\eqref{eq:B-definition}
and~\eqref{eq:T-entries},
$(\BB_k)_{ij}=(-1)^{k+i+j}x^{2k-i-j-2}(x-1)^2$. The $(i,j)$-entry
of the embedded matrix is $(\BB_{k-1})_{i-1,j-1}$. Since
$(i-1)+(j-1)=i+j-2\ge k-1$, applying the same formulas to
$\BB_{k-1}$ gives
$(\BB_{k-1})_{i-1,j-1}=-(-1)^{k+i+j}x^{2k-i-j-2}(x-1)^2
=-(\BB_k)_{ij}$. Thus the two entries cancel.

On the last row or column, let the other index satisfy $1\le i\le k-1$.
By~\eqref{eq:B-definition} and~\eqref{eq:T-boundary}, at position $(k,i)$,
$(\BB_k)_{k,i}=(-1)^i x^{k-i-1}(x-1)$. By the definition of the embedding, the entry of
$\emb_{k,k-1}(\BB_{k-1})$ at position $(k,i)$ is
$(\BB_{k-1})_{k-1,i-1}$, which equals
$(-1)^{i-1}x^{k-i-1}(x-1)=-(\BB_k)_{k,i}$ by the same formulas.
Hence these two entries cancel. By symmetry, the same holds at position
$(i,k)$.

Finally, by~\eqref{eq:B-definition} and~\eqref{eq:T-boundary}, at position
$(k,k)$ we have $(\BB_k)_{k,k}=(-1)^k$, while the embedded entry is
$(\BB_{k-1})_{k-1,k-1}=(-1)^{k-1}$.
These two entries also cancel. This proves
\eqref{eq:B-recurrence}.

For $k=1$, since $\BB_0(x)=[1]$, equation~\eqref{eq:B-recurrence} gives
$\BB_1(x)=(x-1)\PP_1-\emb_{1,0}([1])=\HH_1(x-1)$, so~\eqref{eq:B-decomposition} holds.

Now let $k\ge2$. Substitute~\eqref{eq:B-recurrence} for $\BB_{k-1}$.
The first two terms combine as
\begin{equation}\label{eq:two-step}
\begin{aligned}
 \BB_k(x)
 &=x^{k-2}(x-1)
       \bigl(x\PP_k-\emb_{k,k-1}(\PP_{k-1})\bigr)
       +\emb_{k,k-2}(\BB_{k-2}(x))\\
 &=x^{k-2}(x-1)\HH_k(x)
       +\emb_{k,k-2}(\BB_{k-2}(x)).
\end{aligned}
\end{equation}
Iterating~\eqref{eq:two-step} reduces the index by two at each step.
When $k$ is odd, the iteration terminates at
$\BB_1(x)=\HH_1(x-1)$. When $k$ is even, it terminates at
$\BB_0(x)=[1]$, and $\emb_{k,0}(\BB_0)=\emb_{k,0}([1])=\EE_k$.
\end{proof}

For later use, for two symmetric matrices $A,D$ and a fixed
cut of $D\otimes A$ define
\begin{equation}\label{eq:Q-general}
 \Q(A,D;\zeta,\xi)
 =\cut_{D\otimes A}(\zeta,\xi)-h(A)h(D).
\end{equation}
We suppress $\zeta,\xi$ when the cut is fixed. The expression is linear
in each matrix argument.

\section{Cut bounds for products of path matrices}\label{sec:path-cuts}

\subsection{Short edges and 3-chords}

\paragraph{Path representation.}
We represent $\PP_j$ by the path
\[
0-1-\cdots-(2j+1).
\]
Indeed, as $\ell$ ranges from $0$ to $j$, place column $\ell$ at
position $2\ell+1$; as $i$ ranges from $0$ to $j$, place row $i$ at
position $2(j-i)$. Together these positions are exactly
$0,1,\ldots,2j+1$. By~\eqref{eq:path-matrix}, for each $i$, row $i$ is adjacent precisely to
columns $j-i-1$ and $j-i$, whenever these indices lie in range. Since
column $\ell$ is placed at position $2\ell+1$, the corresponding column
vertices are located at $2(j-i)-1$ and $2(j-i)+1$.
Since row $i$ is placed at position $2(j-i)$, these are precisely the
two neighbouring positions of the row vertex on the path, with the
out-of-range case omitted at the endpoints. As $i$ and $\ell$ range
over $0,\ldots,j$, respectively, these row and column vertices, together
with the adjacencies above, give exactly this path. We call the edges of this path the \emph{short edges}.
The endpoints $0$ and $2j+1$ correspond to row $j$ and column $j$,
respectively, and hence are exactly the two terminals.

\begin{lemma}\label{lem:H-structure}
For $j\ge1$ and real $c$, the edges of $\HH_j(c)$ in the above path
representation are
\begin{equation}\label{eq:H-edges}
\begin{array}{c|c|c}
\text{endpoints}&\text{range}&\text{weight}\\ \hline
(2u,2u+1)&u=0,j&c\\
(2u,2u+1)&1\le u\le j-1&c-1\\
(2u+1,2u+2)&0\le u\le j-1&c\\
(2u,2u+3)&0\le u\le j-1&-1.
\end{array}
\end{equation}
Thus, in the above path representation, $\HH_j(c)$ has the same
short-edge positions as $\PP_j$, with weight $c$ or $c-1$ assigned to
each such position. Those of weight $c$ are called S-edges, while those
of weight $c-1$ are called I-edges; when $j=1$, there are no I-edges.
Accordingly, we also call a short edge of $\PP_j$ type~$S$ or type~$I$
according as the corresponding short-edge position in $\HH_j(c)$ has
weight $c$ or $c-1$, respectively. In addition, $\HH_j(c)$ has the $3$-chords $(2u,2u+3)$,
$0\le u\le j-1$, each of weight $-1$.
Moreover, every $S$-edge lies in exactly one path associated with a $3$-chord, whereas
every $I$-edge lies in exactly two such paths.
\end{lemma}

\begin{proof}
By~\eqref{eq:H-definition}, $\HH_j(c)=c\PP_j-\emb_{j,j-1}(\PP_{j-1})$.
By~\eqref{eq:path-matrix},
\begin{equation}\label{eq:P-support}
(\PP_j)_{i\ell}=1
\quad\Longleftrightarrow\quad
i+\ell=j-1
\ \text{or}\ 
i+\ell=j.
\end{equation}

For the embedded term, its zeroth row and zeroth column are zero, while
for $1\le i,\ell\le j$,
$\bigl(\emb_{j,j-1}(\PP_{j-1})\bigr)_{i\ell}
=(\PP_{j-1})_{i-1,\ell-1}$. For $1\le i,\ell\le j$, this entry equals $1$ if and only if
$(i-1)+(\ell-1)=j-2$ or $(i-1)+(\ell-1)=j-1$, which is equivalent to
\begin{equation}\label{eq:embedded-support}
i+\ell=j
\quad\text{or}\quad
i+\ell=j+1.
\end{equation}

We now compare~\eqref{eq:P-support} and~\eqref{eq:embedded-support}.

If $i+\ell=j-1$, only $c\PP_j$ contributes. With $u=\ell$, the
corresponding row and column vertices lie at $2(j-i)=2u+2$ and
$2\ell+1=2u+1$, respectively. Thus these are the short edges
$(2u+1,2u+2)$, $0\le u\le j-1$, each of weight $c$.

If $i+\ell=j$, the corresponding short edge is $(2u,2u+1)$ with
$u=\ell$. For $1\le u\le j-1$, both $c\PP_j$ and the embedded term
contribute, so the weight is $c-1$.
For $u=0$ or $u=j$, one of $i,\ell$ is zero, so the position lies
outside the embedded block and the weight remains $c$.

Finally, if $i+\ell=j+1$, then $(\PP_j)_{i\ell}=0$, whereas the
embedded term equals $1$. With $u=\ell-1$, the corresponding vertices
lie at $2(j-i)=2u$ and $2\ell+1=2u+3$. Hence these positions give the
$3$-chords $(2u,2u+3)$, $0\le u\le j-1$, each of weight $-1$. This
proves~\eqref{eq:H-edges}.

It remains to count the paths containing each short edge. Two consecutive $3$-chords $(2u,2u+3)$ and $(2u+2,2u+5)$ have the
respective paths $2u-(2u+1)-(2u+2)-(2u+3)$ and
$(2u+2)-(2u+3)-(2u+4)-(2u+5)$, which share the short edge $(2u+2,2u+3)$. Hence every internal edge
$(2v,2v+1)$, $1\le v\le j-1$, namely every type~I edge, lies in exactly two consecutive paths of this form. The remaining short edges are precisely the
type~S edges, and each lies in exactly one such path. This proves the final
assertion.
\end{proof}

\paragraph{Chord replacement.}
Let $\chi_e$ be $1$ when the endpoints of an edge $e$ have different
signs and $0$ otherwise. For a path with edges $e_1,e_2,e_3$ and the
chord $e$ joining its endpoints,
\begin{equation}\label{eq:three-step}
 \chi_{e_1}+\chi_{e_2}+\chi_{e_3}-\chi_e\in\{0,2\}.
\end{equation}
Indeed, if $\chi_e=0$, then the endpoints of the path have
the same sign, so
$\chi_{e_1}+\chi_{e_2}+\chi_{e_3}\in\{0,2\}$. If $\chi_e=1$, then the
endpoints have opposite signs, so
$\chi_{e_1}+\chi_{e_2}+\chi_{e_3}\in\{1,3\}$. Hence, in either case,
\eqref{eq:three-step} holds.

If a chord has weight $-w$, $w>0$, then
\begin{equation}\label{eq:negative-replacement}
 -w\chi_e
 =-w\sum_{f\in P(e)}\chi_f
    +w\left(\sum_{f\in P(e)}\chi_f-\chi_e\right),
\end{equation}
where $P(e)$ is the path joining its endpoints. The second term is nonnegative
by~\eqref{eq:three-step}. Thus the contribution $-w\chi_e$ of the
negative chord can be rewritten by subtracting $w$ from the weight of
each short edge in $P(e)$, together with the separate nonnegative
correction $w\left(\sum_{f\in P(e)}\chi_f-\chi_e\right)$.

\begin{lemma}\label{lem:single-path}
For cuts separating the endpoints of a single path,
\begin{equation*}
 \min\cut_{\HH_j(c)}=
 \begin{cases}
 c-3,&j\ge2,\ c\ge3,\\
 c-1,&j=1,\ c\ge1.
 \end{cases}
\end{equation*}
In particular, $\cut_{\HH_j(c)}-h(\HH_j(c))\ge-2$ when $j\ge2$
and $c\ge3$.
\end{lemma}
\begin{proof}
Apply the above chord replacement to all negative $3$-chords in $\HH_j(c)$.
By Lemma~\ref{lem:H-structure}, the weight of each $S$-edge is reduced
once, whereas the weight of each $I$-edge is reduced twice. Hence the resulting short-edge weights are respectively $c-1$ and $c-3$.

Any cut separating the two endpoints crosses at least one short edge.
Hence, if $j\ge2$ and $c\ge3$, $\cut_{\HH_j(c)}\ge c-3$.
To attain this bound, choose a type~I edge and assign one sign to all
positions before it and the opposite sign to all positions after it.
Then exactly this short edge crosses the cut. Moreover, each such path contains at most one sign change, so for every
chord $e$, $\sum_{f\in P(e)}\chi_f=\chi_e$.
Thus all correction terms in~\eqref{eq:negative-replacement} vanish,
and the cut weight is exactly $c-3$.

If $j=1$, there are no type~I edges, and every modified short edge has
weight $c-1$. Hence every endpoint-separating cut has weight at least
$c-1$. Isolating either endpoint crosses exactly one short edge and
makes the correction term vanish, so the minimum is $c-1$.

Finally, by~\eqref{eq:path-row-sums}, $h(\HH_j(c))=c-1$, and therefore,
for $j\ge2$, $\cut_{\HH_j(c)}-h(\HH_j(c))\ge(c-3)-(c-1)=-2$.
\end{proof}

\subsection{The product of two paths}

For $\PP_k\otimes\PP_j$, reverse the row indices in each coordinate
and write the row and column vertices as
\[
 r_{u,v}\quad(0\le u\le j,\ 0\le v\le k),\qquad
 c_{x,y}\quad(0\le x\le j,\ 0\le y\le k),
\]
where $u$ and $x$ are the fast coordinates, in accordance with the
ordering convention introduced in Section~\ref{sec:reduction}.
Thus $r_{u,v}$ corresponds to the original row $(j-u,k-v)$, while
$c_{x,y}$ corresponds to column $(x,y)$. The corresponding entry of $\PP_k\otimes\PP_j$ is
$(\PP_k\otimes\PP_j)_{(j-u,k-v),(x,y)}
=(\PP_k)_{k-v,y}(\PP_j)_{j-u,x}$. Hence it is nonzero precisely when
$(\PP_j)_{j-u,x}\ne0$ and $(\PP_k)_{k-v,y}\ne0$. By
~\eqref{eq:path-matrix}, these conditions are equivalent to
$x\in\{u-1,u\}$ and $y\in\{v-1,v\}$, respectively.
Let $\Gamma_{j,k}$ be the corresponding bipartite graph, whose two parts consist of the row vertices and the column vertices, respectively. A row vertex $r_{u,v}$ is adjacent to a column vertex $c_{x,y}$ if and only if the corresponding entry of $\PP_k\otimes\PP_j$ is nonzero. By the calculation above, its adjacency rule is
\begin{equation}\label{eq:grid-adjacency}
 r_{u,v}\sim c_{x,y}
 \quad\Longleftrightarrow\quad
 x\in\{u-1,u\},\quad y\in\{v-1,v\},
\end{equation}
with indices outside their ranges discarded. The terminals are
$r_{0,0}$ and $c_{j,k}$, corresponding to the last row and the last
column of $\PP_k\otimes\PP_j$, respectively.

By the definition of $\Gamma_{j,k}$, if
$r_{u,v}c_{x,y}$ is an edge of $\Gamma_{j,k}$, then the corresponding
entry of $\PP_k\otimes\PP_j$ is nonzero. Hence
$(\PP_j)_{j-u,x}\neq0$ and $(\PP_k)_{k-v,y}\neq0$.
Under the path representations of $\PP_j$ and $\PP_k$, these two nonzero
entries correspond to short edges of the respective paths. If these short edges
are of types $A$ and $B$, respectively, then $r_{u,v}c_{x,y}$ is called an
$AB$-edge. Thus the possible types are
$SS$, $SI$, $IS$, and $II$, where the first letter records the type of the short
edge corresponding to $(\PP_j)_{j-u,x}$, and the second records the type of the
short edge corresponding to $(\PP_k)_{k-v,y}$.

\begin{lemma}\label{lem:product-path}
For $j,k\ge1$, the graph $\Gamma_{j,k}$ is connected and its only
bridges are
\begin{equation}\label{eq:terminal-bridges}
 r_{0,0}c_{0,0}\quad\text{and}\quad r_{j,k}c_{j,k}.
\end{equation}
Consequently, every nonterminal cut separating the terminals crosses
at least two edges of $\Gamma_{j,k}$.
\end{lemma}
\begin{proof}
We first show that $\Gamma_{j,k}$ is connected. If $0\le u<j$, then by~\eqref{eq:grid-adjacency},
$r_{u,v}\sim c_{u,v}\sim r_{u+1,v}$, so $r_{u,v}$ and $r_{u+1,v}$ are joined through $c_{u,v}$.

Similarly, if $0\le v<k$, then
$r_{u,v}\sim c_{u,v}\sim r_{u,v+1}$, so $r_{u,v}$ and $r_{u,v+1}$ are joined through $c_{u,v}$.

Starting from $r_{0,0}$ and repeatedly using these two types of connections, we can reach every row vertex $r_{u,v}$. Hence all row vertices lie in the same connected component.

For every column vertex $c_{x,y}$, the adjacency rule~\eqref{eq:grid-adjacency} gives $c_{x,y}\sim r_{x,y}$. Since all row vertices lie in the same connected component, so does every column vertex. Hence $\Gamma_{j,k}$ is connected.

Consider an edge $r_{u,v}c_{x,y}$. Suppose first that $x<j$ and
$v>0$. By~\eqref{eq:grid-adjacency}, $x\in\{u-1,u\}$ and
$y\in\{v-1,v\}$, or equivalently, $u\in\{x,x+1\}$ and
$v\in\{y,y+1\}$.
Let $u'$ be the other member of $\{x,x+1\}$ besides $u$, and let $y'$
be the other member of $\{v-1,v\}$ besides $y$. We claim that
\[
 r_{u,v},\ c_{x,y},\ r_{u',v},\ c_{x,y'},\ r_{u,v}
\]
is a four-cycle. The first edge is the original one. Taking $c_{x,y}$ as the column vertex,
the conditions $u'\in\{x,x+1\}$ and $v\in\{y,y+1\}$ hold by the
definition of $u'$ and the original adjacency, so
$c_{x,y}\sim r_{u',v}$. Next, $x\in\{u'-1,u'\}$ and
$y'\in\{v-1,v\}$, so $r_{u',v}\sim c_{x,y'}$. Finally,
$u\in\{x,x+1\}$ and $v\in\{y',y'+1\}$, and hence $c_{x,y'}\sim r_{u,v}$. Thus the displayed vertices form a
four-cycle containing the original edge.

If $y<k$ and $u>0$, interchange the two coordinates in this
construction. When both conditions fail, the adjacency rule
forces the edge to be one of~\eqref{eq:terminal-bridges}. The two
terminal vertices have degree one, so those edges are indeed bridges.
Finally, a cut with just one crossing edge in a connected graph
crosses a bridge. Here it must isolate one of the degree-one terminals,
and is therefore a terminal cut.
\end{proof}

\begin{lemma}\label{lem:mixed-path}
Throughout this lemma, the cuts under consideration separate the two terminals of the corresponding Kronecker product. If $j,k\ge1$, then
\begin{equation}\label{eq:mixed-path-bound}
 \Q(\PP_j,\PP_k)\ge0.
\end{equation}
Also, $\Q(\PP_j,\HH_k(d))\ge0$ if either $k=1,d\ge4$, or
$k\ge2,d\ge5$. The same statement holds with the factors interchanged.
\end{lemma}
\begin{proof}
We first consider $\PP_k\otimes\PP_j$.
The nonzero entries of $\PP_k\otimes\PP_j$ are all equal to $1$, and its corresponding bipartite graph is $\Gamma_{j,k}$. By
Lemma~\ref{lem:product-path}, $\Gamma_{j,k}$ is connected. Hence every
cut separating the two terminals crosses at least one edge, and therefore
$\cut_{\PP_k\otimes\PP_j}\ge1=h(\PP_j)h(\PP_k)$.
This proves~\eqref{eq:mixed-path-bound}.

We next consider $\HH_k(d)\otimes\PP_j$.
By Lemma~\ref{lem:H-structure}, $\HH_k(d)$ has the same short-edge positions as $\PP_k$, with modified short-edge weights, together with $3$-chords of weight $-1$. Hence the edges arising from the short-edge positions of $\HH_k(d)$ and $\PP_j$ still form the bipartite graph $\Gamma_{j,k}$, with modified weights. The remaining edges arise from the $3$-chord positions of $\HH_k(d)$, together with nonzero entries of the $\PP_j$-factor.

We next characterise the edges of $\HH_k(d)\otimes\PP_j$ whose $\HH_k(d)$-factor corresponds to a $3$-chord of weight $-1$. By Lemma~\ref{lem:H-structure}, in
$\HH_k(d)$, the $3$-chords of weight $-1$ correspond to matrix positions whose
row and column indices sum to $k+1$. In the coordinates of
Section~4.2, the $\HH_k(d)$-row index corresponding to $r_{u,v}$ is
$k-v$. Thus $(k-v)+y=k+1$ if and only if $y=v+1$. On the other hand,
the $\PP_j$-factor is nonzero precisely when $x\in\{u-1,u\}$.
Consequently, the edges arising from these $3$-chords, and only these
edges, have the form
\[
 r_{u,v}c_{x,v+1},
 \qquad
 x\in\{u-1,u\},
 \qquad
 0\le v\le k-1,
\]
with indices outside their ranges discarded.

Fix one of these edges. We claim that
\[
 r_{u,v}-c_{x,v}-r_{u,v+1}-c_{x,v+1}
\]
is a path joining the same two endpoints. Indeed, since $x\in\{u-1,u\}$, equation~\eqref{eq:grid-adjacency}
gives $r_{u,v}\sim c_{x,v}$. Also, $u\in\{x,x+1\}$ and
$v+1\in\{v,v+1\}$, so
$c_{x,v}\sim r_{u,v+1}$. Finally, $x\in\{u-1,u\}$ and
$v+1\in\{v,v+1\}$ give $r_{u,v+1}\sim c_{x,v+1}$. Thus this path has
the same endpoints as the edge $r_{u,v}c_{x,v+1}$. We call this the replacement path associated with
$r_{u,v}c_{x,v+1}$. Therefore~\eqref{eq:negative-replacement}, with
$w=1$, applies to this edge.

Apply this replacement to every such edge. In the sum of these identities,
the weight of each edge of $\Gamma_{j,k}$ is reduced by one for each
replacement path containing it. Thus it remains
only to determine how many replacement paths contain each edge of
$\Gamma_{j,k}$.

Fix $u$ and $x\in\{u-1,u\}$. For $0\le t\le k-1$, write
\[
 R_t:\qquad
 r_{u,t}-c_{x,t}-r_{u,t+1}-c_{x,t+1}.
\]
Its three edges are
\[
 r_{u,t}c_{x,t},
 \qquad
 r_{u,t+1}c_{x,t},
 \qquad
 r_{u,t+1}c_{x,t+1}.
\]

First consider
\[
 r_{u,v}c_{x,v-1},
 \qquad
 1\le v\le k.
\]
In the path representation of $\HH_k(d)$, row $k-v$ is placed at
position $2v$, while column $v-1$ is placed at position
$2(v-1)+1=2v-1$.
Thus the corresponding $\HH_k(d)$-short edge is
$(2v-1,2v)$, which is a type~$S$ edge by Lemma~\ref{lem:H-structure}.
Moreover, $r_{u,v}c_{x,v-1}$ occurs exactly as the middle edge of
$R_{v-1}$, and it cannot occur in any other $R_t$. Hence every edge of
$\Gamma_{j,k}$ of this form is used exactly once.

Next consider
\[
 r_{u,v}c_{x,v}.
\]
For the edge $r_{u,v}c_{x,v}$, the corresponding short edge in the path representation of $\HH_k(d)$ is $(2v,2v+1)$. By
Lemma~\ref{lem:H-structure}, this is a type~$S$ edge when $v=0$ or $v=k$,
and a type~$I$ edge when $1\le v\le k-1$.

If $v=0$, the edge $r_{u,0}c_{x,0}$ occurs only as the first edge of
$R_0$. If $v=k$, the edge $r_{u,k}c_{x,k}$ occurs only as the last edge
of $R_{k-1}$. Thus these two boundary type~$S$ edges are also used
exactly once. Finally, if $1\le v\le k-1$, then
$r_{u,v}c_{x,v}$ is the first edge of $R_v$ and the last edge of
$R_{v-1}$. It is not the middle edge of any replacement path.

Thus the above counting shows that every edge of $\Gamma_{j,k}$ corresponding to an $S$-edge in the path representation of $\HH_k(d)$ is used once, whereas every edge corresponding to an $I$-edge is used twice. Since the corresponding nonzero entry of $\PP_j$ is $1$, these edges have weights $d$ and $d-1$, respectively, before the replacements. Hence their resulting weights are $d-1$ and $d-3$, respectively, and all correction terms arising from
\eqref{eq:negative-replacement} are nonnegative.

Suppose first that $k\ge2$. By Lemma~\ref{lem:product-path}, every
nonterminal cut separating the two terminals crosses at least two
edges of $\Gamma_{j,k}$. Every remaining edge of $\Gamma_{j,k}$ has weight
at least $d-3$. Therefore, when $d\ge5$,
$\cut_{\HH_k(d)\otimes\PP_j}\ge2(d-3)\ge d-1$ for every nonterminal
cut. For either terminal cut, the cut sum is the last row or column sum of
$\HH_k(d)\otimes\PP_j$. By the row-sum property of the Kronecker product and~\eqref{eq:path-row-sums},
$\cut_{\HH_k(d)\otimes\PP_j}=h(\PP_j)h(\HH_k(d))=d-1$. Hence
$\Q(\PP_j,\HH_k(d))\ge0$.

If $k=1$, there are no type~$I$ short edges in $\HH_1(d)$. Thus every
remaining edge of $\Gamma_{j,1}$ has weight $d-1$. Since $\Gamma_{j,1}$ is
connected by Lemma~\ref{lem:product-path}, every cut separating the
terminals crosses at least one such edge. Therefore
$\cut_{\HH_1(d)\otimes\PP_j}\ge d-1=h(\PP_j)h(\HH_1(d))$, and hence $\Q(\PP_j,\HH_1(d))\ge0$.

Finally, interchanging the two coordinates merely relabels the same
weighted bipartite graph and preserves the two terminals. Therefore
the same conclusion holds with the factors interchanged.
\end{proof}

We next need a more detailed count when the path representations of both factors contain $3$-chords.
Delete the type~II edges from $\Gamma_{j,k}$ and denote the resulting
graph by $\Gamma'_{j,k}$. Call the SI- and IS-edges \emph{mixed edges}.

\begin{lemma}
In the spanning subgraph of $\Gamma_{j,k}$ consisting of the SS-edges,
the connected components are indexed by
$(\rho,\sigma)\in\{1,\ldots,j\}\times\{1,\ldots,k\}$.
We call these the SS-components. More precisely, define
\begin{equation}\label{eq:component-labels}
 \lambda(r_{u,v})=(\max(1,u),\max(1,v)),\qquad
 \lambda(c_{x,y})=(\min(j,x+1),\min(k,y+1)).
\end{equation}
Two vertices belong to the same SS-component if and only if they have
the same label.
\end{lemma}
\begin{proof}
For an edge $r_{u,v}c_{x,y}$, the corresponding short edge in $\PP_j$
is of type~S precisely when either
$x=u-1$, or
$x=u$ and $u\in\{0,j\}$. In either case,
$\max(1,u)=\min(j,x+1)$.

Similarly, the corresponding short edge in $\PP_k$ is of type~S
precisely when either
$y=v-1$, or
$y=v$ and $v\in\{0,k\}$. In either case,
$\max(1,v)=\min(k,y+1)$.
Therefore every SS-edge joins two vertices with the same label
in~\eqref{eq:component-labels}.

Conversely, fix
$(\rho,\sigma)\in\{1,\ldots,j\}\times\{1,\ldots,k\}$.
The row vertices with label $(\rho,\sigma)$ are exactly the vertices
$r_{u,v}$ for which
\[
 u=\begin{cases}\rho,&\rho>1,\\0\text{ or }1,&\rho=1,\end{cases}
 \qquad
 v=\begin{cases}\sigma,&\sigma>1,\\0\text{ or }1,&\sigma=1.\end{cases}
\]
Every such row vertex is joined by an SS-edge to
$c_{\rho-1,\sigma-1}$.

Likewise, the column vertices with label $(\rho,\sigma)$ are exactly
the vertices $c_{x,y}$ for which
\[
 x=\begin{cases}\rho-1,&\rho<j,\\j-1\text{ or }j,&\rho=j,\end{cases}
 \qquad
 y=\begin{cases}\sigma-1,&\sigma<k,\\k-1\text{ or }k,&\sigma=k.\end{cases}
\]
Every such column vertex is joined by an SS-edge to $r_{\rho,\sigma}$.

Since $r_{\rho,\sigma}c_{\rho-1,\sigma-1}$ is itself an SS-edge,
all vertices with label $(\rho,\sigma)$ lie in the same SS-component.
Together with the preceding paragraph, this shows that the SS-components
are exactly the label classes in~\eqref{eq:component-labels}.
\end{proof}

\begin{lemma}\label{lem:contracted-grid}
After contracting the SS-components of $\Gamma'_{j,k}$, the mixed edges
form a rectangular grid on
$\{1,\ldots,j\}\times\{1,\ldots,k\}$, possibly with parallel edges.
For $1\le \rho\le j$ and $1\le \sigma\le k-1$, the number of SI-edges joining
the SS-components labelled $(\rho,\sigma)$ and $(\rho,\sigma+1)$ is
\begin{equation}\label{eq:multiplicity}
 \mu_j(\rho)=
 \begin{cases}
 3,&j=1,\\
 2,&j\ge2,\ \rho\in\{1,j\},\\
 1,&1<\rho<j.
 \end{cases}
\end{equation}
Similarly, for $1\le \rho\le j-1$ and $1\le \sigma\le k$, the number of
IS-edges joining $(\rho,\sigma)$ and $(\rho+1,\sigma)$ is $\mu_k(\sigma)$.
In particular, $\Gamma'_{j,k}$ is connected.
\end{lemma}
\begin{proof}
Fix $1\le \rho\le j$ and $1\le \sigma\le k-1$, and consider an SI-edge
$r_{u,v}c_{x,y}$ joining the SS-components labelled $(\rho,\sigma)$ and
$(\rho,\sigma+1)$. By \eqref{eq:grid-adjacency} and
Lemma~\ref{lem:H-structure}, the short edge corresponding to the
$\PP_k$-entry is of type~I precisely when $y=v=t$ for some
$t\in\{1,\ldots,k-1\}$. By \eqref{eq:component-labels}, the second coordinates of the endpoint labels are then $t$ and $t+1$, respectively. Hence
$t=\sigma$.
Thus the corresponding I-edge in $\PP_k$ is fixed, and it remains only
to count the possible pairs $(u,x)$ for which the corresponding short
edge in $\PP_j$ is of type~S and both endpoint labels have first coordinate $\rho$.

If $j\ge2$, these pairs are $(\rho,\rho-1)$, together with $(0,0)$
when $\rho=1$ and $(j,j)$ when $\rho=j$. Hence
$\mu_j(\rho)=2$ for $\rho\in\{1,j\}$ and $\mu_j(\rho)=1$ for
$1<\rho<j$. If $j=1$, the three possible pairs are $(0,0)$, $(1,0)$,
and $(1,1)$, for all of which the endpoint labels have first coordinate $1$, so $\mu_1(1)=3$.
This proves \eqref{eq:multiplicity}.

The IS-edges are counted in the same way after interchanging the roles
of $j$ and $k$. Hence the number joining $(\rho,\sigma)$ and
$(\rho+1,\sigma)$ is $\mu_k(\sigma)$.

Thus every pair of horizontally or vertically adjacent labels is joined
by at least one mixed edge. Conversely, by~\eqref{eq:grid-adjacency}
and~\eqref{eq:component-labels}, every mixed edge joins two labels that agree in
one coordinate and differ by one in the other. After contracting the
SS-components, the underlying simple graph is therefore the rectangular
grid on $\{1,\ldots,j\}\times\{1,\ldots,k\}$, and hence is connected. Since
each contracted vertex represents a connected SS-component, it follows
that $\Gamma'_{j,k}$ itself is connected.
\end{proof}

\begin{lemma}\label{lem:SS-bridges}
Among the type~SS edges of $\Gamma'_{j,k}$, the only bridges are the
two terminal edges
\begin{equation*}
r_{0,0}c_{0,0},\qquad r_{j,k}c_{j,k}.
\end{equation*}
\end{lemma}
\begin{proof}
If $j=1$, then $\mathsf P_j$ has no type~I edges, while if $k=1$, then
$\mathsf P_k$ has no type~I edges. Hence, if $j=1$ or $k=1$, there are no
type~II edges in $\Gamma_{j,k}$, and therefore
$\Gamma'_{j,k}=\Gamma_{j,k}$. The result then follows from
Lemma~\ref{lem:product-path}. Thus we may assume $j,k\ge2$.

For an edge $r_{u,v}c_{x,y}$ of $\Gamma_{j,k}$, the corresponding short
edge in $\mathsf P_j$ is of type~S precisely when $x=u-1$, or when
$x=u$ with $u\in\{0,j\}$; similarly, the corresponding short edge in
$\mathsf P_k$ is of type~S precisely when $y=v-1$, or when $y=v$ with
$v\in\{0,k\}$. We call the type~S edges arising from $x=u$ with
$u\in\{0,j\}$ in $\mathsf P_j$, and from $y=v$ with $v\in\{0,k\}$ in
$\mathsf P_k$, the \emph{boundary S edges}.

For $1\le u\le j$ and $1\le v\le k$, we call
$r_{u,v}c_{u-1,v-1}$ the \emph{core SS-edge} of the component labelled $(u,v)$. All other
type~SS edges will be referred to as \emph{noncore SS-edges}.

We first consider the core SS-edges. If $u\ge2$ and
$v\le k-1$, this edge lies on the six-cycle
\begin{equation*}
 r_{u,v},\ c_{u-1,v},\ r_{u-1,v+1},\ c_{u-2,v},\
 r_{u-1,v},\ c_{u-1,v-1},\ r_{u,v}.
\end{equation*}
Each consecutive pair satisfies~\eqref{eq:grid-adjacency}; the six edge
types are, in order, $SI,IS,SS,SI,IS,SS$.
Hence the entire cycle lies in $\Gamma'_{j,k}$. If instead
$u\le j-1$ and $v\ge2$, the core edge lies on
\begin{equation*}
 r_{u,v},\ c_{u,v-1},\ r_{u+1,v-1},\ c_{u,v-2},\
 r_{u,v-1},\ c_{u-1,v-1},\ r_{u,v},
\end{equation*}
whose edge types are $IS,SI,SS,IS,SI,SS$.
Again the whole cycle is contained in $\Gamma'_{j,k}$. These two
cases cover all core labels except $(1,1)$ and $(j,k)$. The two
remaining core edges lie, respectively, on the four-cycles
\begin{equation*}
\begin{aligned}
 &r_{1,1},\ c_{0,0},\ r_{0,1},\ c_{0,1},\ r_{1,1},\\
 &r_{j,k},\ c_{j-1,k-1},\ r_{j,k-1},\ c_{j,k-1},\ r_{j,k}.
\end{aligned}
\end{equation*}
Every consecutive pair in these cycles satisfies
\eqref{eq:grid-adjacency}, and every edge has at most one type~I
coordinate. Thus both cycles lie in $\Gamma'_{j,k}$. Consequently,
no core SS-edge is a bridge.

It remains to consider the noncore SS-edges. By the preceding classification
of type~S edges in $\mathsf P_j$ and $\mathsf P_k$, every noncore type~SS
edge uses a boundary type~S edge in at least one of the two factors. Hence
the noncore type~SS edges consist of the four families
\begin{equation*}
\begin{aligned}
 &r_{0,v}c_{0,v-1},
 &&1\le v\le k,\\
 &r_{j,v}c_{j,v-1},
 &&1\le v\le k,\\
 &r_{u,0}c_{u-1,0},
 &&1\le u\le j,\\
 &r_{u,k}c_{u-1,k},
 &&1\le u\le j,
\end{aligned}
\end{equation*}
together with the four corner edges $r_{0,0}c_{0,0}$,
$r_{0,k}c_{0,k}$, $r_{j,0}c_{j,0}$, and $r_{j,k}c_{j,k}$.

Each edge in the first family lies on the four-cycle
\[
 r_{0,v}-c_{0,v-1}-r_{1,v}-c_{0,v}-r_{0,v}.
\]
By \eqref{eq:grid-adjacency}, all four adjacencies hold. Moreover, the
short edge in $\mathsf P_j$ corresponding to each edge of this cycle is
of type~S, so none of its edges is of type~II. Hence the cycle lies in
$\Gamma'_{j,k}$.

The other three families are treated in the same way, using the four-cycles
\begin{equation*}
\begin{aligned}
 &r_{j,v}-c_{j,v-1}-r_{j,v-1}-c_{j-1,v-1}-r_{j,v},
 &&1\le v\le k,\\
 &r_{u,0}-c_{u-1,0}-r_{u,1}-c_{u,0}-r_{u,0},
 &&1\le u\le j,\\
 &r_{u,k}-c_{u-1,k}-r_{u-1,k}-c_{u-1,k-1}-r_{u,k},
 &&1\le u\le j.
\end{aligned}
\end{equation*}
Thus every edge in the four families lies on a cycle in $\Gamma'_{j,k}$.

These cycles also contain the two nonterminal corner edges
$r_{0,k}c_{0,k}$ and $r_{j,0}c_{j,0}$. Therefore the only noncore
type~SS edges not yet shown to lie on a cycle are $r_{0,0}c_{0,0}$ and
$r_{j,k}c_{j,k}$.
By \eqref{eq:grid-adjacency}, $r_{0,0}$ and $c_{j,k}$ both have degree
one. Hence these two edges are bridges, completing the proof.
\end{proof}

\subsection{Two factors with negative 3-chords}

The following is the estimate that will be used in the transform
decomposition.

\begin{theorem}\label{thm:path-blocks}
Let each of $A,D$ be one of the matrices
\begin{equation*}
 \PP_j\ (j\ge1),\qquad
 \HH_1(c)\ (c\ge4),\qquad
 \HH_j(c)\ (j\ge2,\ c\ge5).
\end{equation*}
Every cut separating the terminals of $D\otimes A$ satisfies
\begin{equation}\label{eq:path-block-bound}
 \cut_{D\otimes A}\ge h(A)h(D).
\end{equation}
If both factors are H matrices, equality holds only for the two
terminal cuts. Furthermore, if $j,k\ge2$ and $c,d\ge5$, every
nonterminal cut separating the terminals satisfies
\begin{equation}\label{eq:path-block-gap}
 \cut_{\HH_k(d)\otimes\HH_j(c)}-(c-1)(d-1)\ge5.
\end{equation}
\end{theorem}
\begin{proof}
The cases $\PP_j\otimes\PP_k$, $\PP_j\otimes\HH_k(d)$, and
$\HH_k(d)\otimes\PP_j$ are covered by Lemma~\ref{lem:mixed-path}.
We therefore consider $\HH_k(d)\otimes\HH_j(c)$.

We use the same row and column indexing as in Section~4.2. By
Lemma~\ref{lem:H-structure}, $\HH_j(c)$ and $\HH_k(d)$ have the same
short-edge positions as $\PP_j$ and $\PP_k$, respectively. Hence the
edges corresponding to the short-edge positions in both factors still
form the bipartite graph $\Gamma_{j,k}$, with modified weights.

Now consider an edge for which a $3$-chord position occurs in exactly
one of the two factors. Suppose first that the $3$-chord position lies
in $\HH_k(d)$, while the position in $\HH_j(c)$ is a short-edge
position. Since the short-edge positions of $\HH_j(c)$ are exactly the
nonzero positions of $\PP_j$, these negative edges and their replacement
paths occur in exactly the same positions as those already obtained for
$\HH_k(d)\otimes\PP_j$ in Lemma~\ref{lem:mixed-path}.

More precisely, such a negative edge has the form
\[
 r_{u,v}c_{x,v+1},
 \qquad
 x\in\{u-1,u\},
 \qquad
 0\le v\le k-1,
\]
and its replacement path is
\[
 r_{u,v}-c_{x,v}-r_{u,v+1}-c_{x,v+1}.
\]
The only difference from the $\HH_k(d)\otimes\PP_j$ case is the weight.
If the corresponding short edge of $\HH_j(c)$ has weight $\alpha$, then
the negative edge has weight $-\alpha$, and \eqref{eq:negative-replacement}
is applied to the same replacement path with $w=\alpha$.

Similarly, after interchanging the two factors, if the $3$-chord position
lies in $\HH_j(c)$ while the position in $\HH_k(d)$ is a short-edge
position, then the corresponding negative edges and replacement paths are
obtained in the same way. Such a negative edge has the form
\[
 r_{u,v}c_{u+1,y},
 \qquad
 0\le u\le j-1,
 \qquad
 y\in\{v-1,v\},
\]
with replacement path
\[
 r_{u,v}-c_{u,y}-r_{u+1,v}-c_{u+1,y}.
\]
If the short edge of $\HH_k(d)$ corresponding to the position $(k-v,y)$
has weight $\beta$, then the negative edge has weight $-\beta$, and
\eqref{eq:negative-replacement} is applied with $w=\beta$.

Finally, if the corresponding positions in both factors are $3$-chord
positions, then the resulting edge has weight $(-1)(-1)=1$.

Apply the above replacements to all negative edges, and fix an edge
$f=r_{u,v}c_{x,y}$ of $\Gamma_{j,k}$.

Consider first the replacements arising from $3$-chord positions of
$\HH_k(d)$. As observed above, their replacement paths are exactly the
same as those in the $\HH_k(d)\otimes\PP_j$ case of
Lemma~\ref{lem:mixed-path}. Hence, by the counting already carried out
there, $f$ is contained in exactly one such replacement path if its
corresponding short edge in the path representation of $\HH_k(d)$ is of
type~S, and in exactly two if that short edge is of type~I.

Similarly, after interchanging the two factors, the same counting rule
holds for the short edge of $\HH_j(c)$ corresponding to $f$.

Let $\alpha$ and $\beta$ be the weights of the short edges corresponding
to $f$ in the path representations of $\HH_j(c)$ and $\HH_k(d)$,
respectively. Define
\[
\begin{aligned}
 m_1&=\begin{cases}
  1,&\text{if the corresponding short edge of $\HH_j(c)$ is of type~S},\\
  2,&\text{if it is of type~I},
 \end{cases}\\
 m_2&=\begin{cases}
  1,&\text{if the corresponding short edge of $\HH_k(d)$ is of type~S},\\
  2,&\text{if it is of type~I}.
 \end{cases}
\end{aligned}
\]

For a replacement arising from a $3$-chord in one factor, the
corresponding matrix position in the other factor is the same for all
three edges of the replacement path.

Before the replacements, the weight of $f$ is $\alpha\beta$. Each
replacement arising from a $3$-chord of $\HH_j(c)$ reduces the weight of
$f$ by $\beta$, and there are $m_1$ such replacements containing $f$.
Likewise, each replacement arising from a $3$-chord of $\HH_k(d)$
reduces the weight of $f$ by $\alpha$, and there are $m_2$ such
replacements containing $f$. Therefore the remaining weight of $f$ is
\begin{equation}\label{eq:residual-weight}
 \alpha\beta-m_1\beta-m_2\alpha.
\end{equation}
Apply~\eqref{eq:negative-replacement} to every negative edge and
sum the resulting exact identities. Let $R$ be the weighted graph with
underlying graph $\Gamma_{j,k}$, whose edge weights are the remaining
weights given by~\eqref{eq:residual-weight}. Then
\begin{equation}\label{eq:two-H-cut-decomposition}
 \cut_{\HH_k(d)\otimes\HH_j(c)}
 =\cut_R+\sum_{e\text{ positive}}\chi_e
 +\sum_{e\text{ negative}}|w_e|
       \left(\sum_{f\in P(e)}\chi_f-\chi_e\right).
\end{equation}
For each negative edge $e$, $w_e$ denotes its original weight and $P(e)$ is its associated replacement path.
Here the first sum is over the positive edges whose corresponding
positions in both factors are $3$-chord positions; each such edge has
weight one. Each correction term in the second sum is nonnegative
by~\eqref{eq:three-step}. Hence
$\cut_{\HH_k(d)\otimes\HH_j(c)}\ge\cut_R$.
We first treat the three minimal parameter cases.

\medskip\noindent
\emph{Both indices at least two: $j,k\ge2$ and $c=d=5$.}
By~\eqref{eq:residual-weight}, the edge weights in $R$ are
\begin{equation*}
\begin{array}{c|ccc}
 \text{type}&\mathrm{SS}&\mathrm{SI}\text{ or }\mathrm{IS}&\mathrm{II}\\ \hline
 \text{weight}&25-5-5=15&20-10-4=6&16-8-8=0.
 \end{array}
\end{equation*}
In this case, $h(\HH_j(5))h(\HH_k(5))=(5-1)^2=16$, so it remains to prove that every terminal-separating cut has weight at least $16$.

For nonterminal cuts, the contributions of the SS-edges and the mixed edges of $R$ suffice for the required lower bound, so we restrict attention to these edges. Since the underlying graph of $R$ is $\Gamma_{j,k}$, the SS-edges and the mixed edges of $R$ form precisely $\Gamma'_{j,k}$.

If a cut crosses at least two SS-edges, its weight in $R$ is at
least $30$. If it crosses an SS-edge and a mixed edge, the weight is
at least $15+6=21$. If it crosses exactly one SS-edge and no mixed
edge, that edge is the only edge of $\Gamma'_{j,k}$ crossing the cut. By
Lemma~\ref{lem:contracted-grid}, this graph is connected, so the edge
is a bridge. Lemma~\ref{lem:SS-bridges} then shows that this is a
terminal cut.

Now suppose no SS-edge crosses the cut. Then each SS-component lies
entirely on one side of the cut, so these components may be contracted
without losing information about the separation of the two terminals.
Hence it remains only to count mixed SI/IS edges between distinct SS-components.

By Lemma~\ref{lem:contracted-grid}, after contracting the SS-components,
the mixed edges form a rectangular grid on the component labels. The
terminal components have labels $(1,1)$ and $(j,k)$, and these labels
carry opposite signs. Consider the two boundary paths
\[
 (1,1),(2,1),\ldots,(j,1),(j,2),\ldots,(j,k)
\]
and
\[
 (1,1),(1,2),\ldots,(1,k),(2,k),\ldots,(j,k).
\]
The two paths are edge-disjoint, and each has endpoints of opposite
signs. Hence each path contains a pair of adjacent labels with opposite
signs, giving two distinct such boundary pairs. By~\eqref{eq:multiplicity},
each of these boundary pairs is joined by exactly two mixed edges in
$\Gamma'_{j,k}$. Therefore at least four mixed edges cross the original cut, and hence
$\operatorname{cut}_R\ge4\cdot6=24$.

Together with~\eqref{eq:two-H-cut-decomposition}, these estimates show
that every nonterminal cut of $\HH_k(5)\otimes\HH_j(5)$ has weight at
least $21=16+5$. A terminal cut of this Kronecker product has weight
$h(\HH_j(5))h(\HH_k(5))=16$ by~\eqref{eq:path-row-sums}. This proves both the equality assertion and~\eqref{eq:path-block-gap}
at $(c,d)=(5,5)$.

\medskip\noindent
\emph{One index equals one: $j=1,k\ge2$ and $c=4,d=5$.}
By Lemma~\ref{lem:H-structure}, $\HH_1(4)$ has no type~$I$ short
edge. Hence every edge of $R$ is of type~$SS$ or $SI$. By \eqref{eq:residual-weight}, their weights are respectively
$4\cdot5-5-4=11$ and $4\cdot4-4-8=4$. In this case,
$h(\HH_1(4))h(\HH_k(5))=(4-1)(5-1)=12$,
so it remains to prove that every terminal-separating cut has weight at
least $12$.

Since the underlying graph of $R$ is $\Gamma_{1,k}$, its $SS$-edges
and mixed edges form precisely $\Gamma'_{1,k}$.

If a cut crosses at least two $SS$-edges, then its weight in $R$ is at
least $2\cdot11=22$. If it crosses an $SS$-edge together with a mixed
edge, then its weight in $R$ is at least $11+4=15$.
If it crosses exactly one $SS$-edge and no mixed edge, then that edge
is the only edge of $\Gamma'_{1,k}$ crossing the cut; by
Lemmas~\ref{lem:contracted-grid} and \ref{lem:SS-bridges}, the cut is
terminal.

Now suppose that no $SS$-edge crosses the cut. Since $j=1$, the
contracted rectangular grid of Lemma~\ref{lem:contracted-grid} reduces
to the path $(1,1),(1,2),\ldots,(1,k)$. The two endpoint labels represent
the SS-components containing the terminals and therefore have opposite
signs; hence at least one sign change occurs along this path.
By \eqref{eq:multiplicity}, each pair of consecutive labels is joined by
three mixed edges. Hence, if there are at least two sign changes along
this path, then $\operatorname{cut}_R\ge2\cdot3\cdot4=24$.

It remains to consider the case of exactly one sign change, say between
$(1,h)$ and $(1,h+1)$, where $1\le h<k$. The three mixed edges joining
these two labels contribute $3\cdot4=12$ to $\operatorname{cut}_R$.

To obtain strict inequality in this remaining case, consider the edge
$r_{0,h-1}c_{1,h}$. Here the column coordinates satisfy $1=0+1$ and
$h=(h-1)+1$.
By the characterisation of the $3$-chord positions above, the
corresponding positions in both $\HH_1(4)$ and $\HH_k(5)$ are
$3$-chord positions. Hence this is one of the additional positive
edges appearing in the first sum of
\eqref{eq:two-H-cut-decomposition}, and it has weight one.

By the label formulas in \eqref{eq:component-labels}, its endpoints
belong to the $SS$-components labelled $(1,\max\{1,h-1\})$ and
$(1,h+1)$. Because the unique sign change occurs between $h$ and $h+1$,
and $\max\{1,h-1\}\le h$,
these two components have opposite signs. Thus this additional positive
edge also crosses the cut. Together with the three mixed edges above,
\[
 \operatorname{cut}_{\HH_k(5)\otimes\HH_1(4)}
 \ge12+1=13>12.
\]

Terminal cuts have weight $h(\HH_1(4))h(\HH_k(5))=12$ by \eqref{eq:path-row-sums}. The case $k=1,j\ge2$ follows by
interchanging the two coordinates.

\medskip\noindent
\emph{Both indices equal one: $j=k=1$ and $c=d=4$.}
By Lemma~\ref{lem:H-structure}, in the path representations of both
factors, every short edge is of type~$S$. Hence every edge of $R$ is of type $SS$, and by
\eqref{eq:residual-weight} its weight is $4\cdot4-4-4=8$. In this case,
$h(\HH_1(4))^2=(4-1)^2=9$,
so it remains to prove that every terminal-separating cut has weight at
least $9$.

Since the underlying graph of $R$ is $\Gamma_{1,1}$,
Lemma~\ref{lem:product-path} implies that every nonterminal cut separating
the two terminals crosses at least two edges of $R$. Therefore
$\operatorname{cut}_R\ge2\cdot8=16>9$.
By \eqref{eq:two-H-cut-decomposition}, the weight of the corresponding
cut of $\HH_1(4)\otimes\HH_1(4)$ is at least $\operatorname{cut}_R$.
Hence every nonterminal cut has weight strictly greater than $9$, whereas
the terminal cuts have weight $h(\HH_1(4))^2=9$ by \eqref{eq:path-row-sums}.

\medskip
To pass from the minimal admissible parameters to arbitrary $c$ and $d$,
let $c_0=4$ if $j=1$ and $c_0=5$ otherwise, and define $d_0$ in the same
way. By the definition of $\HH_j(c)$,
$\HH_j(c)=\HH_j(c_0)+(c-c_0)\PP_j$, and similarly for $\HH_k(d)$. Since $\Q$ is bilinear in its matrix arguments for each fixed cut
by~\eqref{eq:Q-general},
\begin{equation}\label{eq:parameter-expansion}
\begin{aligned}
 \Q(\HH_j(c),\HH_k(d))
 ={}&\Q(\HH_j(c_0),\HH_k(d_0))\\
 &+(c-c_0)\Q(\PP_j,\HH_k(d_0))\\
 &+(d-d_0)\Q(\HH_j(c_0),\PP_k)\\
 &+(c-c_0)(d-d_0)\Q(\PP_j,\PP_k).
\end{aligned}
\end{equation}
Here $c-c_0,d-d_0\ge0$, and the last three terms are nonnegative by
Lemma~\ref{lem:mixed-path}. Hence
$\Q(\HH_j(c),\HH_k(d))\ge\Q(\HH_j(c_0),\HH_k(d_0))$.
By the three minimal-parameter cases proved above, the right-hand side is
strictly positive for every nonterminal cut. On the other hand, terminal
cuts have $\Q=0$ for all parameters by~\eqref{eq:path-row-sums}.
Therefore equality in~\eqref{eq:path-block-bound} holds exactly for the
terminal cuts. When $j,k\ge2$, the base-case bound
$\Q(\HH_j(5),\HH_k(5))\ge5$ is preserved by~\eqref{eq:parameter-expansion},
proving~\eqref{eq:path-block-gap}. This proves all assertions.
\end{proof}

\section{The two maximising divisor sets}\label{sec:proof-main}

We now apply Theorem~\ref{thm:path-blocks} to the decomposition of
$\BB_a(p)$ and $\BB_b(q)$. Odd exponents give only H matrices.
Even exponents also give a last-coordinate term, whose possible loss
will be bounded using Lemma~\ref{lem:single-path}.

\begin{proposition}\label{prop:odd-odd}
Let $a,b$ be positive odd integers and let $p,q\ge5$ be real.
Every cut separating the last row and column vertices of
$\BB_b(q)\otimes\BB_a(p)$ has weight at least $g_a(p)g_b(q)$.
Equality holds exactly at the two terminal cuts.
\end{proposition}
\begin{proof}
Since $a$ and $b$ are odd,~\eqref{eq:B-decomposition} gives
\[
\begin{aligned}
 \BB_a(p)
 &=\sum_{\substack{j=a,a-2,\ldots\\ j\ge3}}
   p^{j-2}(p-1)\emb_{a,j}(\HH_j(p))
   +\emb_{a,1}(\HH_1(p-1)),\\
 \BB_b(q)
 &=\sum_{\substack{k=b,b-2,\ldots\\ k\ge3}}
   q^{k-2}(q-1)\emb_{b,k}(\HH_k(q))
   +\emb_{b,1}(\HH_1(q-1)).
\end{aligned}
\]
For the indices occurring above, write
\[
\begin{gathered}
 c_j=\begin{cases}p,&j\ge3,\\ p-1,&j=1,\end{cases}
 \qquad
 d_k=\begin{cases}q,&k\ge3,\\ q-1,&k=1,\end{cases}\\
 \alpha_j=\begin{cases}p^{j-2}(p-1),&j\ge3,\\1,&j=1,\end{cases}
 \qquad
 \beta_k=\begin{cases}q^{k-2}(q-1),&k\ge3,\\1,&k=1.\end{cases}
\end{gathered}
\]
All these coefficients are positive. Moreover, for $j\ge3$ the
corresponding factor is $\HH_j(p)$ with $p\ge5$, whereas for $j=1$
it is $\HH_1(p-1)$ with $p-1\ge4$; the same statements hold on the
$q$-side. Thus every pair of H factors occurring in the two
decompositions satisfies the parameter assumptions of
Theorem~\ref{thm:path-blocks}.

Fix a cut separating the last row and column vertices of
$\BB_b(q)\otimes\BB_a(p)$. For a pair of indices $j,k$, the matrices
$\emb_{a,j}(\HH_j(c_j))$ and $\emb_{b,k}(\HH_k(d_k))$ are supported on
the coordinate sets $\{a-j,\ldots,a\}$ and $\{b-k,\ldots,b\}$,
respectively. Hence, in the original row and column coordinates, their
tensor product is supported on the lower-right rectangle
$\{a-j,\ldots,a\}\times\{b-k,\ldots,b\}$.
This rectangle contains the final coordinate $(a,b)$, so its last row
and column vertices are precisely the two global terminals. Therefore
the restriction of the fixed global cut to this rectangle still
separates the two terminals.

The embedded tensor product differs from
$\HH_k(d_k)\otimes\HH_j(c_j)$ only by zero rows and columns outside
this lower-right rectangle, and these zero entries make no contribution
to the cut weight. When the fixed global cut is restricted to this
rectangle and the coordinates are relabelled as
$\{0,\ldots,j\}\times\{0,\ldots,k\}$, the resulting cut weight is exactly
that of the corresponding cut of $\HH_k(d_k)\otimes\HH_j(c_j)$. Hence Theorem~\ref{thm:path-blocks}
applies and gives
\[
 \cut_{\HH_k(d_k)\otimes\HH_j(c_j)}
 \ge h(\HH_j(c_j))h(\HH_k(d_k)).
\]

By bilinearity of the Kronecker product,
\[
 \BB_b(q)\otimes\BB_a(p)
 =\sum_{j,k}\alpha_j\beta_k
 \bigl(\emb_{b,k}(\HH_k(d_k))\otimes
       \emb_{a,j}(\HH_j(c_j))\bigr).
\]
For the fixed cut, the cut sum is linear in the matrix entries. Since
the zero entries introduced by the embeddings do not contribute, the
preceding blockwise bound yields
\[
\begin{aligned}
 \cut_{\BB_b(q)\otimes\BB_a(p)}
 &\ge \sum_{j,k}\alpha_j\beta_k
       h(\HH_j(c_j))h(\HH_k(d_k))\\
 &=\left(\sum_j\alpha_jh(\HH_j(c_j))\right)
   \left(\sum_k\beta_kh(\HH_k(d_k))\right).
\end{aligned}
\]
Since $\emb_{a,j}$ places the matrix in the lower-right corner and all
entries outside the embedded block are zero, the last row sum of the embedded
matrix is the same as that of the original matrix; thus
$h\bigl(\emb_{a,j}(\HH_j(c_j))\bigr)=h(\HH_j(c_j))$. By linearity of
$h$ and the above decompositions,
$h(\BB_a(p))=\sum_j\alpha_jh(\HH_j(c_j))$ and
$h(\BB_b(q))=\sum_k\beta_kh(\HH_k(d_k))$. Consequently,
$\cut_{\BB_b(q)\otimes\BB_a(p)}\ge h(\BB_a(p))h(\BB_b(q))$.
By~\eqref{eq:B-row-sum}, $h(\BB_a(p))=g_a(p)$ and
$h(\BB_b(q))=g_b(q)$, and therefore
\begin{equation}\label{eq:odd-cut-bound}
 \cut_{\BB_b(q)\otimes\BB_a(p)}-g_a(p)g_b(q)\ge0.
\end{equation}

It remains to determine when equality can occur. For the fixed global
cut, bilinearity of $\Q$ and the above decompositions give
\[
\Q(\BB_a(p),\BB_b(q))
=\sum_{j,k}\alpha_j\beta_k
\Q\!\left(
\emb_{a,j}(\HH_j(c_j)),
\emb_{b,k}(\HH_k(d_k))
\right).
\]
For each pair $j,k$, restricting the fixed global cut to the
corresponding lower-right rectangle and relabelling the coordinates
yields a cut of $\HH_k(d_k)\otimes\HH_j(c_j)$ that separates its two
terminals. Since the embeddings introduce only zero rows and columns
outside the blocks and preserve the last row sums,
\[
\Q\!\left(
\emb_{a,j}(\HH_j(c_j)),
\emb_{b,k}(\HH_k(d_k))
\right)
=\Q(\HH_j(c_j),\HH_k(d_k))
\]
for this cut. Hence every term in the sum is nonnegative by
Theorem~\ref{thm:path-blocks}.

Suppose now that the global cut is nonterminal. For $j=a$ and $k=b$,
the embeddings are trivial:
\[
\emb_{a,a}(\HH_a(c_a))=\HH_a(c_a),
\qquad
\emb_{b,b}(\HH_b(d_b))=\HH_b(d_b).
\]
Thus the corresponding term is
\[
\Q(\HH_a(c_a),\HH_b(d_b))
\]
for the original global cut itself. Since this cut is nonterminal, the
equality assertion of Theorem~\ref{thm:path-blocks} gives
\[
\Q(\HH_a(c_a),\HH_b(d_b))>0.
\]
Since $\alpha_a\beta_b>0$, while all the other terms in the sum are
nonnegative, it follows that
\[
\Q(\BB_a(p),\BB_b(q))>0.
\]
Equivalently,
\[
\cut_{\BB_b(q)\otimes\BB_a(p)}>g_a(p)g_b(q)
\]
for every nonterminal cut.

Finally, by~\eqref{eq:B-row-sum}, the last row sums of $\BB_a(p)$ and
$\BB_b(q)$ are $g_a(p)$ and $g_b(q)$, respectively. Hence the last row
sum of their Kronecker product is $g_a(p)g_b(q)$. Since the matrices
are symmetric, the last column sum has the same value. The two terminal
cuts therefore attain equality in~\eqref{eq:odd-cut-bound}.
\end{proof}

\begin{proposition}\label{prop:even-even}
The conclusion of Proposition~\ref{prop:odd-odd} also holds when
$a,b\ge2$ are even and $p,q\ge5$ are real.
\end{proposition}
\begin{proof}
For even $a$, formula~\eqref{eq:B-decomposition} becomes
\[
 \BB_a(p)=\sum_{j=a,a-2,\ldots,2}
              w_j(p)\emb_{a,j}(\HH_j(p))+\EE_a,
 \qquad w_j(p)=(p-1)p^{j-2}.
\]
Let $W_p=\sum_j w_j(p)=(p^a-1)/(p+1)$ and
$\alpha=w_a(p)=(p-1)p^{a-2}$. Then
\[
 \frac{W_p}{\alpha}
 =1+p^{-2}+\cdots+p^{-(a-2)}
 \le\frac1{1-p^{-2}}\le\frac{25}{24},\qquad \alpha\ge4.
\]
Similarly,
$\BB_b(q)=\sum_{k=b,b-2,\ldots,2}w_k(q)\emb_{b,k}(\HH_k(q))+\EE_b$,
and we set $W_q=\sum_k w_k(q)$ and
$\beta=w_b(q)=(q-1)q^{b-2}$. Thus
$W_q\le(25/24)\beta$ and $\beta\ge4$.

Fix a nonterminal cut separating the last row and column vertices of
$\BB_b(q)\otimes\BB_a(p)$. Since $\Q$ is bilinear in its matrix arguments for each fixed cut by
\eqref{eq:Q-general}, the decompositions above give
\begin{align}\label{eq:even-Q-expansion}
 \Q(\BB_a(p),\BB_b(q))
={}&\sum_{j,k}w_j(p)w_k(q)
 \Q\!\left(\emb_{a,j}(\HH_j(p)),\emb_{b,k}(\HH_k(q))\right)\notag\\
&+\sum_j w_j(p)
 \Q\!\left(\emb_{a,j}(\HH_j(p)),\EE_b\right)\notag\\
&+\sum_k w_k(q)
 \Q\!\left(\EE_a,\emb_{b,k}(\HH_k(q))\right)
 +\Q(\EE_a,\EE_b).
\end{align}
We estimate these four parts separately.

Each summand in the first sum in~\eqref{eq:even-Q-expansion}
corresponds to an H--H block pair. As in Proposition~\ref{prop:odd-odd}, the
corresponding lower-right rectangle contains the two global terminals,
and the zero entries introduced by the embeddings make no contribution
to the cut weight. Hence the restriction of the global cut to this
rectangle is still terminal-separating, and Theorem~\ref{thm:path-blocks}
gives $\Q\!\left(\emb_{a,j}(\HH_j(p)),\emb_{b,k}(\HH_k(q))\right)\ge0$.
For the outermost pair $j=a$, $k=b$, the two blocks use all coordinates,
so their tensor product occupies the full coordinate rectangle.
Restricting the global cut to this pair therefore removes no vertex, and
the restricted cut is the global cut itself, hence is still nonterminal.
Since $a,b\ge2$ and $p,q\ge5$, the stronger estimate
\eqref{eq:path-block-gap} gives $\Q(\HH_a(p),\HH_b(q))\ge5$, and the
coefficient of this outermost term is $w_a(p)w_b(q)=\alpha\beta$.
Therefore the total contribution of all H--H terms is at least
\begin{equation}\label{eq:even-HH}
 5\alpha\beta.
\end{equation}

We next consider the mixed terms
$\Q\!\left(\emb_{a,j}(\HH_j(p)),\EE_b\right)$. Since
$\EE_b=e_be_b^{\mathsf T}$, the tensor product
$\EE_b\otimes\emb_{a,j}(\HH_j(p))$ has only one nonzero block, namely $\emb_{a,j}(\HH_j(p))$ in its
lower-right corner. All other blocks are zero and hence make no
contribution to the cut weight. The embedding
$\emb_{a,j}(\HH_j(p))$ itself differs from $\HH_j(p)$ only by zero
rows and columns preceding its lower-right support. After deleting this
zero padding and relabelling the remaining coordinates, the nonzero
part is exactly one copy of $\HH_j(p)$.

Since $\EE_b=e_be_b^{\mathsf T}$, the $q$-coordinates of both the row and column vertices are fixed at $b$. Under the lower-right embedding $\emb_{a,j}$, the last row and last column of $\HH_j(p)$ correspond to the $p$-coordinate $a$. Hence the two terminal vertices of this copy are precisely the two global terminal vertices indexed by $t=(a,b)$. Therefore, restricting the global cut to this copy still separates its two terminals.

Because $j\ge2$ and $p\ge5$, Lemma~\ref{lem:single-path} gives
$\cut_{\HH_j(p)}-h(\HH_j(p))\ge-2$. Moreover, $h(\EE_b)=1$ and
$h\!\left(\emb_{a,j}(\HH_j(p))\right)=h(\HH_j(p))$ because a
lower-right embedding preserves the last row sum. Since the zero padding
also makes no contribution to the cut weight, the corresponding mixed
difference satisfies
$\Q\!\left(\emb_{a,j}(\HH_j(p)),\EE_b\right)
=\cut_{\HH_j(p)}-h(\HH_j(p))\ge-2$.
Multiplying by $w_j(p)$ and summing over $j$ yields
\begin{equation*}
 \sum_j w_j(p)
 \Q\!\left(\emb_{a,j}(\HH_j(p)),\EE_b\right)
 \ge-2W_p.
\end{equation*}

The other mixed terms are handled symmetrically. Since
$\EE_a=e_ae_a^{\mathsf T}$, the tensor product
$\emb_{b,k}(\HH_k(q))\otimes\EE_a$ has, after its zero padding is removed, exactly one nonzero copy of
$\HH_k(q)$. Its two endpoints again coincide with the two global
terminals, so the restricted cut separates them. Lemma~\ref{lem:single-path}
therefore gives
$\Q\!\left(\EE_a,\emb_{b,k}(\HH_k(q))\right)\ge-2$, and hence
\begin{equation*}
 \sum_k w_k(q)
 \Q\!\left(\EE_a,\emb_{b,k}(\HH_k(q))\right)
 \ge-2W_q.
\end{equation*}

Finally, $\EE_b\otimes\EE_a$ has exactly one nonzero entry, of weight
one, joining the two global terminals. Since the fixed cut separates
these terminals, this edge crosses the cut. Also $h(\EE_a)=h(\EE_b)=1$.
Therefore
\begin{equation}\label{eq:even-EE}
 \Q(\EE_a,\EE_b)=1-1=0.
\end{equation}

Combining~\eqref{eq:even-HH}--\eqref{eq:even-EE}, we obtain
\begin{equation}\label{eq:even-cut-estimate}
\begin{aligned}
 \Q(\BB_a(p),\BB_b(q))
 &\ge5\alpha\beta-2W_p-2W_q\\
 &\ge5\alpha\beta-\frac{25}{12}(\alpha+\beta).
\end{aligned}
\end{equation}
Since $\alpha,\beta\ge4$, we have
$\frac1\alpha+\frac1\beta\le\frac12$, and hence
$\alpha+\beta\le\frac{\alpha\beta}{2}$.
Substitution into~\eqref{eq:even-cut-estimate} gives
\begin{equation*}
\begin{aligned}
 \Q(\BB_a(p),\BB_b(q))
 &\ge5\alpha\beta-\frac{25}{24}\alpha\beta\\
 &=\frac{95}{24}\alpha\beta>0.
\end{aligned}
\end{equation*}
Thus every nonterminal cut has weight strictly greater than
$g_a(p)g_b(q)$.

For either terminal cut, the cut sum is the last row or column sum of
$\BB_b(q)\otimes\BB_a(p)$. By~\eqref{eq:B-row-sum},
$h(\BB_a(p))=g_a(p)$ and $h(\BB_b(q))=g_b(q)$, so the last row sum of
the Kronecker product is
$h(\BB_b(q))h(\BB_a(p))=g_b(q)g_a(p)$.
Since the matrices are symmetric, the last column sum has the same
value. Hence both terminal cuts attain $g_a(p)g_b(q)$, and equality
holds exactly for the two terminal cuts.
\end{proof}

\begin{proof}[Proof of Theorem~\ref{thm:main}]
Because $a+b$ is even, the two exponents are either both odd or both
even. Propositions~\ref{prop:odd-odd} and~\ref{prop:even-even} prove
that every cut in Lemma~\ref{lem:energy-cut} has weight at least
$\gamma=g_a(p)g_b(q)$, with equality exactly at the two terminal cuts.
That lemma gives
\[
 \max_{y_t=-1}F(y)
 =p^aq^b+d_a(p)d_b(q)-2g_a(p)g_b(q),
\]
attained only by $y=-s$ and $y=s-2e_t$. Dividing by two in
\eqref{eq:F} proves~\eqref{eq:main-energy}.

For $y=-s$, equation~\eqref{eq:centred-vector} gives
$x_{i,j}=(1-(-1)^{i+j})/2$, so $x$ selects precisely $\Dodd$.
For $y=s-2e_t$, it gives
$x=(\one+s)/2-e_t$, which selects the even checkerboard and removes
$n$; this is $\Deven$. Both sets contain proper divisors only: since $a+b$ is even, $n$ is
absent from the first set and removed from the second. They are nonempty, because $p,q\in\Dodd$
and $1\in\Deven$. The characterisation of the maximising vectors $y$ therefore gives
exactly the two divisor sets in~\eqref{eq:maximising-sets}.
\end{proof}

\section*{Acknowledgments}
This work was supported by the National Natural Science Foundation of China under Grant No.~12071351.

\section*{Declaration of generative AI and AI-assisted technologies in the manuscript preparation process}
During the preparation of this manuscript, the authors used ChatGPT to assist in drafting and refining the presentation of mathematical statements, improving the clarity of the writing, and suggesting candidate auxiliary lemmas within the authors' established research approach. The central mathematical ideas, main theorems, and overall proof strategy were developed by the authors. Any AI-suggested auxiliary lemmas incorporated into the manuscript were selected and independently verified by the authors. All AI-assisted content was carefully reviewed and edited, and the authors take full responsibility for the manuscript.

\end{document}